\documentclass[12pt, a4paper, twoside]{article}
\usepackage{amsmath,amsfonts,amssymb,amsthm,amscd}
\usepackage{epsfig}
\numberwithin{equation}{section}
\newtheorem{theorem}{Theorem}[section]

\newtheorem{lemma}[theorem]{Lemma}
\newtheorem{recurrence}[theorem]{Recurrence}
\newtheorem{corollary}[theorem]{Corollary}
\newtheorem{observation}[theorem]{Observation}
\newtheorem{problem}{Problem}[section]
\newtheorem*{remark*}{Remark}
\newtheorem*{keywords*}{Keywords}

\long\def\onefigure#1#2#3#4{
\begin{figure}[ht]
\begin{center}
\scalebox{#2}
{
\includegraphics{#1}
}
\end{center}
\caption{\label{#4} #3}
\end{figure}
}

\newcommand{\myfig}[3]
{\onefigure{{#1.eps}} {#2} {#3} {fig:#1}}

\DeclareMathOperator{\rmex}{ex}
\DeclareMathOperator{\MST}{MST}
\DeclareMathOperator{\SMT}{SMT}
\DeclareMathOperator{\AFF}{AFF}
\DeclareMathOperator{\DFF}{DFF}

\def\inst#1{$^{#1}$}

\begin{document}

\title{Tight bounds on the maximum size of a set of permutations with bounded VC-dimension
\thanks{
This is the authors' version of a work that was accepted for publication in 
Journal of Combinatorial Theory, Series A. \newline
An extended abstract of this paper appeared in Proceedings of the Twenty-Third Annual 
ACM-SIAM Symposium on Discrete Algorithms (SODA'12), SIAM, 2012, pp. 1113-–1122. \newline
Research was supported by the project CE-ITI (GACR P2020/12/G061) 
of the Czech Science Foundation,
by project no.\ 52410 of the Grant Agency of Charles University and
by the grant SVV-2012-265313 (Discrete Methods and Algorithms).
Josef Cibulka was also supported by the Czech Science Foundation under 
the contract no.\ 201/09/H057.
} 
} 

\author{Josef Cibulka\inst{1}, Jan Kyn\v{c}l\inst{2}
} 

\date{}

\maketitle

\begin{center}
{\footnotesize
\inst{1} 
Department of Applied Mathematics, \\
Charles University, Faculty of Mathematics and Physics, \\
Malostransk\'e n\'am.~25, 118~ 00 Praha 1, Czech Republic; \\ 
\texttt{cibulka@kam.mff.cuni.cz} 
\\\ \\
\inst{2}
Department of Applied Mathematics and Institute for Theoretical Computer Science, \\
Charles University, Faculty of Mathematics and Physics, \\
Malostransk\'e n\'am.~25, 118~ 00 Praha 1, Czech Republic; \\
\texttt{kyncl@kam.mff.cuni.cz}
}
\end{center}  

\begin{abstract}
The \emph{VC-dimension} of a family $\mathcal P$ of $n$-permutations is the largest integer $k$
such that the set of restrictions of the permutations in $\mathcal P$ on some $k$-tuple of positions
is the set of all $k!$ permutation patterns.
Let $r_k(n)$ be the maximum size of a set of $n$-permutations with VC-dimension $k$.
Raz showed that $r_2(n)$ grows exponentially in $n$. 
We show that $r_3(n)=2^{\Theta(n \log \alpha(n))}$ and for every $t\geq 1$, we have 
$r_{2t+2}(n) = 2^{\Theta(n \alpha(n)^t)}$ and $r_{2t+3}(n) = 2^{O(n \alpha(n)^{t} \log \alpha(n))}$.

We also study the maximum number $p_k(n)$ of $1$-entries in an $n\times n$ $(0,1)$-matrix 
with no $(k+1)$-tuple of columns containing all $(k+1)$-permutation matrices. 
We determine that, for example, $p_3(n) = \Theta(n\alpha(n))$ and $p_{2t+2}(n) = n 2^{(1/t!)\alpha(n)^t \pm O(\alpha(n)^{t-1})}$ 
for every $t \geq 1$.

We also show that for every positive $s$ there is a slowly growing function $\zeta_s(n)$ 
(for example $\zeta_{2t+3}(n) = 2^{O(\alpha^{t}(n))}$ for every $t \ge 1$) satisfying the following.
For all positive integers $n$ and $B$ and every $n \times n$ $(0,1)$-matrix $M$ with $\zeta_s(n)Bn$ $1$-entries,
the rows of $M$ can be partitioned into $s$ intervals so that at least $B$ columns contain at least $B$ 
$1$-entries in each of the intervals.
\end{abstract}

\begin{keywords*}
\rm
permutation pattern, VC-dimension, Davenport--Schinzel sequence, set of permutations, inverse Ackermann function
\end{keywords*}

\section{Introduction}
Let $\mathcal T$ be a set system on $[n]=\{1,2,\dots,n\}$. We say that a set $K\subset [n]$
is \emph{shattered} by $\mathcal T$ if every subset of $K$ appears as an intersection of $K$ and 
some set from $\mathcal T$.
The \emph{Vapnik--Chervonenkis dimension (VC-dimension)} of $\mathcal T$ is the size of the largest
set shattered by $\mathcal T$.
Sauer's lemma gives the exact value of the maximum size of a set system on $[n]$ with VC-dimension $k$, 
which is a polynomial in $n$ of degree $k$.
More on the VC-dimension and its history can be found for example in~\cite{LDG}.

Motivated by the so-called acyclic linear orders problem, Raz~\cite{Raz00} defined the VC-dimension of a set 
$\mathcal P$ of permutations:
Let $S_n$ be the set of all \emph{$n$-permutations}, that is, permutations of $[n]$.
The \emph{restriction of $\pi \in S_n$ to the $k$-tuple $(a_1, a_2, \dots, a_k)$ of positions} 
(where $1\leq a_1<a_2< \dots <a_k\leq n$) is the $k$-permutation 
$\pi'$ satisfying $\forall i,j: \pi'(i) < \pi'(j) \Leftrightarrow \pi(a_i) < \pi(a_j)$.
The $k$-tuple of positions $(a_1, \dots, a_k)$ is \emph{shattered by $\mathcal P$} if each
$k$-permutation appears as a restriction of some $\pi \in \mathcal P$ to $(a_1, \dots, a_k)$.
The \emph{VC-dimension of $\mathcal P$} is the size of the largest set of positions shattered by $\mathcal P$.
Let $r_k(n)$ be the size of the largest set of $n$-permutations with VC-dimension $k$. 

Raz~\cite{Raz00} proved that $r_2(n)\leq C^n$ for some constant $C$ and asked 
whether an exponential upper bound on $r_k(n)$ can also be found for every $k\geq 3$.

An $n$-permutation $\pi$ \emph{avoids} a $k$-permutation $\rho$ if none of the restrictions of $\pi$ to a $k$-tuple 
of positions is $\rho$. Clearly, the set of permutations avoiding $\rho \in S_k$ has VC-dimension smaller than $k$.
Thus, Raz's question generalizes the Stanley--Wilf conjecture which states that the number of $n$-permutations 
that avoid an arbitrary fixed permutation $\rho$ grows exponentially in $n$. The conjecture was settled by Marcus 
and Tardos~\cite{MarcusTardos} using a result of Klazar~\cite{Klazar00}.

We show in Section~\ref{sec:ub} that the size of a set of $n$-permutations with VC-dimension $k$ cannot 
be much larger than exponential in $n$. The result has an application in enumerating simple complete 
topological graphs~\cite{Kyncl}.
Let $\alpha(n)$ be the inverse of the Ackermann function; see Section~\ref{sub:fff} for its definition.

\begin{theorem}
\label{thm:vcub}
The sizes of sets of permutations with bounded VC-dimension satisfy
\begin{align*}
r_3(n) &\leq \alpha(n)^{(4+o(1))n}, \\
r_4(n) &\leq 2^{n \cdot \left(2 \alpha(n) + 3 \log_2(\alpha(n)) + O(1) \right)}, \\
r_{2t+2}(n) &\leq 2^{n \cdot ((2/t!)\alpha(n)^t + O(\alpha(n)^{t-1}) )} \qquad \textrm{for every $t \geq 2$ and} \\
r_{2t+3}(n) &\leq 2^{n \cdot ((2/t!)\alpha(n)^t \log_2(\alpha(n)) + O(\alpha(n)^t) } \qquad \textrm{for every $t \geq 1$.} \\
\end{align*}
\end{theorem}

On the other hand, we give a negative answer to Raz's question in Section~\ref{sec:lb}.
\begin{theorem}
\label{thm:vclb}
We have
\begin{align*}
r_3(n) &\ge (\alpha(n)/2-O(1))^{n} \qquad \textrm{and} \\
r_{2t+3}(n) \ge r_{2t+2}(n) &\ge 2^{n \cdot ((1/t!)\alpha(n)^t - O(\alpha(n)^{t-1}) )} \qquad \textrm{for every $t \geq 1$.} \\
\end{align*}
\end{theorem}

An \emph{$n$-permutation matrix} is an $n \times n$ $(0,1)$-matrix with exactly one $1$-entry in every row and 
every column. Permutations and permutation matrices are in a one-to-one correspondence that assigns to a permutation 
$\pi$ a permutation matrix $A_{\pi}$ with $A_{\pi}(i,j)=1 \Leftrightarrow \pi(j)=i$.

An $m\times n$ $(0,1)$-matrix $B$ \emph{contains} a $k\times l$ $(0,1)$-matrix $S$ if $B$ has a $k\times l$
submatrix $T$ that can be obtained from $S$ by changing some (possibly none) $0$-entries to $1$-entries. 
Otherwise $B$ \emph{avoids} $S$. Thus, a permutation $\pi$ avoids $\rho$ if and only if $A_{\pi}$ avoids $A_{\rho}$.
F\"{u}redi and Hajnal~\cite{FurediHajnal} studied the following problems from the extremal theory of
$(0,1)$-matrices. Given a matrix $S$ (the \emph{forbidden matrix}), what is the maximum number 
$\rmex_S(n)$ of $1$-entries 
in an $n\times n$ matrix that avoids $S$? This area is closely related to Tur\'{a}n problems on graphs and to 
Davenport--Schinzel sequences. Functions $\rmex_S$ or their asymptotics have been determined for some matrices 
$S$~\cite{FurediHajnal, PettieFH, Tardos05} and these results have found applications mostly in discrete 
geometry~\cite{BienstockGyori, EfratSharir, Furedi90, PachTardos} and also in the analysis 
of algorithms~\cite{PettieAppl}. 
The F\"{u}redi--Hajnal conjecture states that $\rmex_P(n)$ is linear in $n$ whenever $P$ is a permutation matrix.
Marcus and Tardos proved this conjecture by a surprisingly simple argument~\cite{MarcusTardos}.
This implied the relatively long standing Stanley--Wilf conjecture by Klazar's reduction~\cite{Klazar00}.
An improved reduction yielding the upper bound $2^{O(k \log k)n}$ on the size of a set of $n$-permutations 
with a forbidden $k$-permutation was found by the first author~\cite{Cibulka}.

We modify the question of F\"{u}redi and Hajnal and study the maximal number $p_k(n)$ of $1$-entries in an 
$n \times n$ matrix such that no $(k+1)$-tuple of columns contains all $(k+1)$-permutation matrices. 
It can be easily shown that $p_2(n) = 4n-4$. Indeed, consider an $n \times n$ matrix with at least $4n-3$
$1$-entries. Remove the highest and the lowest $1$-entry in every column. Then the first and the last
row of the resulting matrix contain no $1$-entry and thus one of the rows contains three $1$-entries. 
The three columns of the original matrix containing these $1$-entries contain every $3$-permutation matrix.
The lower bound $4n-4$ can be achieved for example by filling the two top rows and some two columns with $1$'s.

\begin{theorem}
\label{thm:extremal}
We have
\begin{align*}
2 n \alpha(n) - O(n) \leq p_3(n) & \leq O(n \alpha(n)),\\
p_{2t+2}(n) &= n 2^{(1/t!) \alpha(n)^t \pm O(\alpha(n)^{t-1}) } \quad \textrm{for every $t \geq 1$ and} \\
n 2^{(1/t!) \alpha(n)^t - O(\alpha(n)^{t-1})} \leq p_{2t+3}(n) 
&\leq n 2^{(1/t!) \alpha(n)^t \log_2(\alpha(n)) + O(\alpha(n)^t)} \quad \textrm{for every $t \geq 1$.}
\end{align*}
\end{theorem}

The upper bounds from Theorem~\ref{thm:extremal} are proven as Corollary~\ref{cor:extrub} in Section~\ref{sub:uones}
and the lower bounds as Corollary~\ref{cor:extrlb} in Section~\ref{sub:lmat}.

Let $S$ and $T$ be sequences. We say that $S$ \emph{contains} a pattern $T$ if $S$ contains a subsequence
$T'$ isomorphic to $T$, that is, $T$ can be obtained from $T'$ by a one-to-one renaming of the symbols.
A sequence $S$ over an alphabet $\Gamma$ is a \emph{Davenport--Schinzel sequence of order $s$} 
(a $\mathrm{DS}(s)$-sequence for short) if no symbol appears on two consecutive positions and $S$ does not 
contain the pattern $abab\dots$ of length $s+2$.
These sequences were introduced by Davenport and Schinzel~\cite{DavenportSchinzel} and found numerous
applications in computational and combinatorial geometry. More can be found in the book of Sharir and
Agarwal~\cite{DSbook}. 
Let $\lambda_s(n)$ be the maximum length of a
Davenport--Schinzel sequence over $n$ symbols. The following are the current best bounds on $\lambda_s(n)$.

\begin{align*}
2n \alpha(n) - O(n) &\leq \lambda_3(n) \leq 2n \alpha(n) + O \left( n \sqrt{\alpha(n)} \right), \\
n \cdot 2^{(1/t!) \alpha(n)^t - O\left(\alpha(n)^{t-1}\right)} 
& \leq \lambda_{2t+2}(n) 
\le n \cdot 2^{(1/t!)\alpha(n)^t + O\left(\alpha(n)^{t-1}\right)} \qquad \textrm{for $t\geq 1$ and}\\
n \cdot 2^{(1/t!) \alpha(n)^t - O\left(\alpha(n)^{t-1}\right)}	
& \leq\lambda_{2t+3}(n) 
\le n \cdot 2^{(1/t!) \alpha(n)^t \log_2 \alpha(n) + O\left(\alpha(n)^t\right)}	
\qquad \textrm{for $t\geq 1$}.
\end{align*}

The upper bound on $\lambda_3$ is by Klazar~\cite{Klazar99}, the lower bounds on $\lambda_s$ for $s>3$ are
by Agarwal, Sharir and Shor~\cite{AgarwalSharirShor} and all the other bounds were proved by Nivasch~\cite{Nivasch10}.

Pettie~\cite{PettieDSTightish} recently announced the following improved bounds:
\begin{align*}
\Omega(n \alpha(n) 2^{\alpha(n)}) \le \lambda_5(n) &\le O(n \alpha^2(n) 2^{\alpha(n)}) \qquad \text{and}\\
\lambda_{2t+3}(n) &\le n \cdot 2^{(1/t!) \alpha(n)^t (1 + o(1)) }
\qquad \textrm{for $t\geq 2$}.
\end{align*}

Our proofs are based on several results on Davenport--Schinzel sequences 
as well as on sequences with other forbidden patterns. 
The results on sequences that we use are mentioned in more detail in
Sections~\ref{sub:uones},~\ref{sub:fff}~and~\ref{sub:lseq},
where they are transformed into claims about matrices with forbidden patterns.

An \emph{$s$-partition of the rows} of an $m \times n$ matrix $M$ is a partition of the interval 
of integers $\{1, \dots, m\}$ into $s$ intervals 
$\{1=m_1,\dots, m_2-1\}$, $\{m_2,\dots, m_3-1\}$, \dots, $\{m_s,\dots, m=m_{s+1}-1\}$.
A matrix $M$ contains a \emph{$B$-fat $(r,s)$-formation} if there exists an $s$-partition
of the rows and an $r$-tuple of columns each of which has $B$ $1$-entries in each interval of rows. 
Note that the order of the columns in the matrix is not important for this notion.
See Fig.~\ref{fig:formation} for an example of a $1$-fat $(3,4)$-formation.
In Section~\ref{sub:fff}, we prove the following lemma, which gives an upper bound on the number 
of $1$-entries an $n \times n$ matrix can have and still not contain any $B$-fat $(B,s)$-formation.
It is used in the proof of Theorem~\ref{thm:vcub} 
in Section~\ref{sub:uperm}, analogously to the use of Raz's Technical Lemma~\cite{Raz00}. 

\begin{lemma}
\label{lem:fatff}
For all positive integers $s,n$ and $B$, an $n \times n$ matrix $M$ with at least $\zeta_s(n) B n$ 
$1$-entries contains a $B$-fat $(B, s)$-formation, where $\zeta_s(n)$ are functions of the form
\begin{align*}
\zeta_2(n) &= O(1), \quad \zeta_3(n) = O(\alpha(n)), \quad \zeta_4(n) = O(\alpha(n)^2), \quad \zeta_5(n) = O(\alpha(n) 2^{\alpha(n)}), \\
\zeta_{2t+3}(n) &= 2^{(1/t!)\alpha(n)^t + O(\alpha(n)^{t-1})} \qquad \textrm{for $t\geq 2$ and} \\
\zeta_{2t+4}(n) &= 2^{(1/t!)\alpha(n)^t \log(\alpha(n)) + O(\alpha(n)^t)} \qquad \textrm{for $t\geq 1$.}
\end{align*}
More generally, for all positive integers $m,n,s$ and $B$, an $m \times n$ matrix $M$ 
with at least $\zeta_s(m) B n$ $1$-entries contains a $B$-fat $(\lfloor nB/m\rfloor, s)$-formation.
\end{lemma}

The proof of the lemma is based on a proof of the upper bound on the number of symbols in
the so-called formation-free sequences (see definition in Section~\ref{sub:uones}) from Nivasch's 
paper~\cite{Nivasch10}.

By an argument similar to the proof of $p_2(n) \le 4n-4$ above, it is easy to verify that every $m \times n$ matrix $M$ 
with at least $3n$ $1$-entries contains a $1$-fat $(\lceil n/m\rceil, 3)$-formation.
A similar result for $2$-fat formations would slightly improve the upper bounds on $r_3(n)$ and $r_4(n)$.

\begin{problem}
Does there exist a constant $c$ such that for every $m$ and $n$, every $m \times n$ matrix $M$ 
with at least $c n$ $1$-entries contains a $2$-fat $(\lfloor n/m\rfloor, 3)$-formation?
\end{problem}

All logarithms in this paper are in base $2$.

\section{Upper bounds}
\label{sec:ub}

\subsection{Numbers of $1$-entries in matrices}
\label{sub:uones}

A sequence $S$ of length $l$ over an alphabet $\Gamma$ is a function $S:[l]\rightarrow \Gamma$.
An \emph{$(r,s)$-formation} is a sequence formed by $s$ concatenated permutations of the same 
$r$-tuple of symbols. The permutations in a formation are its \emph{troops}.
A sequence $S=(a_1, \dots, a_l)$ is \emph{$r$-sparse} if $a_i \neq a_j$ whenever $0 < |i-j| < r$.
An \emph{$(r,s)$-formation-free sequence} is a sequence that is $r$-sparse and contains no 
$(r,s)$-formation as a subsequence. Let $F_{r,s}(n)$ be the maximum length of an 
$(r,s)$-formation-free sequence over $n$ symbols. Formation-free sequences were first studied 
by Klazar~\cite{Klazar92}.

To be able to use results on sequences for matrices, we use the 
\emph{matrix$\rightarrow$sequence transcription $\MST$} (our name) 
defined by Pettie~\cite{PettieFH} who improved an earlier transcription by 
F\"{u}redi and Hajnal~\cite{FurediHajnal}. 
The letters of the sequence correspond to the columns of the matrix.
The matrix is transcribed row by row from top to bottom. Let $\mathrm{Seq}_{i-1}$ be the
sequence created from the first $i-1$ rows. We consider the set $C_i$ of letters 
corresponding to the columns having a $1$-entry in the row $i$. The letters in $C_i$ are ordered 
in the order of the last appearance in $\mathrm{Seq}_{i-1}$; the one that appeared last 
in $\mathrm{Seq}_{i-1}$ is first and so on. The letters 
that did not appear in $\mathrm{Seq}_{i-1}$ are ordered arbitrarily and placed after those
that did appear. The ordered sequence $C_i$ is then appended to $\mathrm{Seq}_{i-1}$.
The length of the resulting sequence $\MST(M) = \mathrm{Seq}_m$ is equal to the 
number of $1$-entries in $M$ and the size of the alphabet is $n$. 
Note that the previous papers~(\cite{FurediHajnal,PettieFH}) transcribe the matrices 
column by column instead of row by row.

A \emph{block} in a sequence is a contiguous subsequence containing only distinct symbols.
Note that $\MST(M)$ can be decomposed into $m$ or fewer blocks.

A set $S$ of $rs$ $1$-entries forms an \emph{$(r,s)$-formation} in $M$ if there exists an $s$-partition
of the rows and an $r$-tuple of columns each of which has a $1$-entry of $S$ in every interval of rows 
of the partition. See Fig.~\ref{fig:formation}. 
In this and all other figures, circles and full circles represent the $1$-entries and empty space
represents the $0$-entries.
A matrix $M$ is \emph{$(r,s)$-formation-free} if it contains no $(r,s)$-formation.

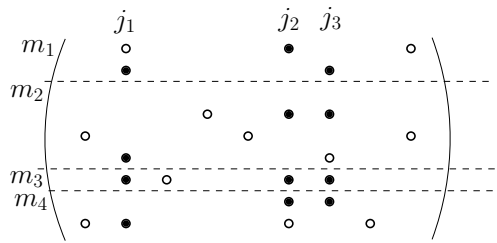
\myfig{formation}{0.6}{A $(3,4)$-formation on columns $j_1$, $j_2$ and $j_3$. Full circles represent 
the $1$-entries of the formation. Empty circles represent $1$-entries outside of this formation.
}

\begin{lemma}
\label{lem:form2spl}
A $(0,1)$-matrix $M$ contains an $(r,s)$-formation if and only if $\MST(M)$ contains an $(r,s)$-formation.
\end{lemma}
\begin{proof}
Observe that an $(r,s)$-formation in a matrix $M$ implies an $(r,s)$-formation in $\MST(M)$. 

The proof of the other direction is more complicated, because symbols of one block of $\MST(M)$ may be present
in two troops of the $(r,s)$-formation in $\MST(M)$. 
To overcome this complication, we consider such an $(r,s)$-formation in $\MST(M)$, whose each troop ends 
earliest possible. Assume that the $i$-th troop ends with an occurrence of a symbol $a$ in the $j$-th
block of $\MST(M)$ and that the $(i+1)$-st troop begins with $b$ from the $j$-th block. 
Since $a$ precedes $b$ in the $j$-th block, we know, by the definition of $\MST(M)$, that $a$ appears 
somewhere between the occurrences of $b$ and $a$ of the $i$-th troop. Therefore, the $i$-th troop could
end earlier, contradicting the selection of the $(r,s)$-formation.
\end{proof}

Nivasch gives the following upper bound on the maximum length $F_{r,s}(n)$ of an $(r,s)$-formation-free 
sequence on $n$ symbols:
\begin{theorem}{\rm (\cite[Theorem 1.3]{Nivasch10})}
\label{thm:nivasch13}
For every $r \in \mathbb{N}$
\[
F_{r,4}(n) \leq O(n \alpha(n)).
\]
For every $r$ and every $s \geq 5$, letting $t:=\lfloor (s-3)/2 \rfloor$, we have
\begin{align*}
F_{r,s}(n) & \leq n 2^{(1/t!)\alpha(n)^t \log(\alpha(n)) + O(\alpha(n)^t)} \qquad  \textrm{when $s$ is even and}\\
F_{r,s}(n) & \leq n 2^{(1/t!)\alpha(n)^t + O(\alpha(n)^{t-1})} \qquad  \textrm{when $s$ is odd}.
\end{align*}
\end{theorem}

Let $p'_k(n)$ be the maximum number of $1$-entries in an $(k+1,k+1)$-formation-free $n \times n$ matrix. Theorem~\ref{thm:nivasch13} implies the following upper bounds on $p'_k(n)$.

\begin{lemma}
\label{lem:formfree}
We have
\[
p'_3(n) \leq O(n \alpha(n)).
\]
 
For every fixed $k\geq 4$, letting $t:=\lfloor (k-2)/2 \rfloor$, we have
\begin{align*}
p'_k(n) & \leq n 2^{(1/t!)\alpha(n)^t \log(\alpha(n)) + O(\alpha(n)^t)}
\qquad  \textrm{when $k$ is odd and} \\
p'_k(n) & \leq n 2^{(1/t!)\alpha(n)^t + O(\alpha(n)^{t-1})} \qquad \textrm{when $k$ is even.}
\end{align*}
\end{lemma}
\begin{proof}
Take a $(k+1,k+1)$-formation-free matrix $M$. Then $\MST(M)$ does not contain 
any $(k+1,k+1)$-formation by Lemma~\ref{lem:form2spl}.

The sequence $\MST(M)=a_1,a_2,\dots,a_p$ can be made $(k+1)$-sparse by removing at most $kn$ 
occurrences of symbols. Indeed, whenever two occurrences $a_i, a_j$ (where $i<j$) of the same symbol 
appear at distance at most $k$, then $a_i$ is among the last $k$ symbols preceding the block containing 
$a_j$. Thus, it suffices to take the blocks from left to right and in each of them remove the at most 
$k$ symbols that appear as the last $k$ symbols preceding the block.
The resulting sequence is thus a $(k+1,k+1)$-formation-free sequence of length differing 
by $O(n)$ from the number of $1$-entries of $M$. The result then follows from Theorem~\ref{thm:nivasch13}.
\end{proof}

This proves the upper bounds in Theorem~\ref{thm:extremal} by observing that a $(k+1)$-tuple of columns with 
a $(k+1,k+1)$-formation contains every $(k+1)$-permutation matrix.
\begin{corollary}
\label{cor:extrub}
For every fixed $k\geq 3$ if we let $t:=\lfloor (k-2)/2 \rfloor$, then
\begin{align*}
p_3(n) & \leq O(n \alpha(n)),  \\
p_k(n) & \leq n 2^{(1/t!)\alpha(n)^t \log(\alpha(n)) + O(\alpha(n)^t)}
\qquad  \textrm{when $k$ is odd and greater than $3$ and} \\
p_k(n) & \leq n 2^{(1/t!)\alpha(n)^t + O(\alpha(n)^{t-1})} \qquad \textrm{when $k$ is even.}
\end{align*}
\end{corollary}

\subsection{Fat formations in matrices}
\label{sub:fff}

A sequence $S$ is an \emph{$\AFF_{r,s,k}(m)$-sequence}\footnote{
AFF is an abbreviation for \emph{almost-formation-free}.}
if it contains no $(r,s)$-formation as a subsequence, can be decomposed into $m$ or fewer 
blocks and each symbol of the sequence appears at least $k$ times. 
Let $\Pi'_{r,s,k}(m)$ be the maximum number of symbols in an $\AFF_{r,s,k}(m)$-sequence.

Let $\alpha_d(m)$ be the $d$th function in the \emph{inverse Ackermann hierarchy}. 
That is, $\alpha_1(m)=\lceil m/2 \rceil$,
$\alpha_d(1)=0$ for $d \geq 2$ and $\alpha_d(m) = 1+ \alpha_d(\alpha_{d-1}(m))$ for $m,d  \geq 2$.
The \emph{inverse Ackermann function} is defined as $\alpha(m) := \min\{k: \alpha_k(m) \leq 3\}$.

Nivasch defines a hierarchy of functions $R_s(d)$, which we shift by $1$ in the index. That is, our $R_s(d)$
is the original $R_{s-1}(d)$. We thus have the functions defined for $s \geq 2$ and $d \geq 2$. 
The values are $R_2(d)=2$, $R_3(d)=3$, $R_4(d)=2d+1$, $R_s(2)=2^{s-2}+1$ and
\[
R_s(d) = 2 (R_{s-1}(d)-1) + (R_{s-2}(d)-1)(R_s(d-1)-3) + 1 \qquad \textrm{when $s \geq 5$ and $d \geq 3$}.
\]
For $s\geq 5$, if we let $t=\lfloor (s-3)/2 \rfloor$, then $R_s(d) = 2^{(1/t!)d^t \log(d) + O(d^t)}$
if $s$ is even and $R_s(d) = 2^{(1/t!)d^t + O(d^{t-1})}$ when $s$ is odd. 

\begin{lemma} {\rm (\cite[Corollary 5.14]{Nivasch10})}
\label{lem:nivasch514}
For every $d\geq 2$, $s \geq 3$, $r\geq 2$, $m$ and $k$ satisfying $m \geq k \geq R_s(d)$ we have
\[
\Pi'_{r,s,k}(m) \leq c'_s r m \alpha_d(m)^{s-3},
\]
where $c'_s$ is a constant depending only on $s$.
\end{lemma}

The linear dependence of the upper bound on $r$ is not explicitly mentioned in~\cite{Nivasch10}, 
but can be revealed from the proof. In the base case, the dependence on $r$ is linear 
(Lemmas 5.9 and 5.10 in~\cite{Nivasch10}) and in Recurrences 5.11 and 5.13, the right-hand side can be 
rewritten as $r$ times an expression not depending on $r$.

It was shown~\cite{KlazarValtr,PettieDbl} that doubling letters in the forbidden subsequence usually 
has small impact on the maximum length of a generalized DS-sequence (see the definition in~\cite{Klazar92}).
Geneson~\cite{Geneson09} generalized 
the linear upper bound from the F\"uredi--Hajnal conjecture to forbidden double permutation matrices.
We show a similar behavior of formation-free sequences and matrices.
For $s\geq 2$, a set $S$ of $r(2s-2)$ $1$-entries forms a \emph{doubled $(r,s)$-formation} in $M$ 
if there exists an $s$-partition
of the rows and an $r$-tuple of columns each of which has one $1$-entry of $S$ in the top and bottom interval 
of rows of the partition and two $1$-entries in every other interval. 
A matrix $M$ is \emph{doubled $(r,s)$-formation-free} if it contains no doubled $(r,s)$-formation.
A \emph{$\DFF_{r,s,k}(m)$-matrix} is a doubled $(r,s)$-formation-free matrix with $m$ rows and at least
$k$ $1$-entries in every column.
Let $\Delta_{r,s,k}(m)$ be the maximum number of columns in a $\DFF_{r,s,k}(m)$-matrix.

In Corollary~\ref{cor:dblffhier} we show an analogue of Lemma~\ref{lem:nivasch514} 
for doubled $(r,s)$-formation-free matrices.
The proof follows the structure of the proof of Corollary~5.14 in~\cite{Nivasch10}. 
First, we show some simple bounds on $\Delta_{r,s,k}(m)$. 
The case $d=2$ of Corollary~\ref{cor:dblffhier} is proved in Corollary~\ref{cor:dblffd2}
by Recurrence~\ref{rec:rec1} and the remaining cases follow from Recurrence~\ref{rec:rec2}.
Corollary~\ref{cor:dblffhier} will give a sequence of upper bounds on $\Delta_{r,s,k}(m)$. 
Typically, the bounds are superlinear in $m$ for $r$, $s$ and $k$ fixed and the subsequence of
bounds applicable is limited by the values of $s$ and $k$. As $k$ grows (keeping $r$ and $s$ fixed) 
the best applicable bound gets closer and closer to linear. When one lets $k$ be a suitable 
function of $\alpha(m)$, the bound becomes linear in $m$.

If $m<k$, no matrix with $m$ rows can have $k$ $1$'s in every column.
\begin{observation}
\label{obs:dsmallm}
For every $r,s,k,m$, if $m<k$, then
\[
\Delta_{r,s,k}(m) = 0.
\]
\end{observation}

\begin{observation}
For every $r,s,k,m$, if $k<2s-2$, then
\[
\Delta_{r,s,k}(m) = \infty.
\]
\end{observation}

Analogously to~\cite[Lemma~5.10]{Nivasch10}, all the other values of $\Delta_{r,s,k}(m)$ are finite. 
\begin{observation}
\label{obs:drpolys}
For every $r \geq 2$, $s \geq 2$ and $m \geq 2s-2$
\[
\Delta_{r,s,2s-2}(m) \leq (r-1) \binom{m-s+1}{s-1} \leq r m^{s-1}.
\]
\end{observation}
\begin{proof}
If each column in an $r$-tuple of columns has the same position of the $2$nd, $4$th, \dots, $(2s-2)$nd 
$1$-entry, then the first $2s-2$ $1$-entries from the columns form a doubled $(r,s)$-formation. 
\end{proof}

\begin{recurrence}
For every $r,k,m$ and $s \geq 3$
\label{rec:rec1}
\[
\Delta_{r,s,2k+1}(2m) \leq 2 \Delta_{r,s,2k+1}(m) + 2 \Delta_{r,s-1,k}(m).
\]
\end{recurrence}
\begin{proof}
As in the proof of \cite[Recurrence~5.11]{Nivasch10}, we cut the rows of a $\DFF_{r,s,2k+1}(2m)$-matrix 
into the upper $m$ rows and the lower $m$ rows. 
The \emph{local columns} are those with all $1$-entries in the same half of rows. 
There are at most $2\Delta_{r,s,2k+1}(m)$ local columns. Columns that are not local are \emph{global}.
Consider the submatrix $M'_1$ formed by the upper half of rows of global columns with at least half of their 
$1$'s in the upper half of rows. Let $M_1$ be the matrix created from $M'_1$ by removing the lowest $1$ 
in every column of $M'_1$. 
If $M_1$ contains a doubled $(r,s-1)$-formation, then $M$ contains a doubled $(r,s)$-formation.
Thus $M_1$ has at most $\Delta_{r,s-1,k}(m)$ columns.
A symmetric argument can be applied on the global columns with at least half of their 
$1$'s in the lower half of rows.
\end{proof}

\begin{corollary}
\label{cor:dblffd2}
For every fixed $s \geq 2$ and for all integers $r,k,m$ satisfying $k \geq 2^{s-1}+2^{s-2}-1$ we have
\[
\Delta_{r,s,k}(m) \leq \bar{c}_s r m \log(m)^{s-2},
\]
where $\bar{c}_s$ is a constant depending only on $s$.
\end{corollary}
\begin{proof}
The proof proceeds by induction on $s$ and $m$.
The base case of $s=2$ follows from Observation~\ref{obs:drpolys} and the cases with $m<k$ from 
Observation~\ref{obs:dsmallm}. Recurrence~\ref{rec:rec1} is used as the induction step.
\end{proof}

\begin{recurrence}
\label{rec:rec2}
For every nonnegative $r,m, k_1, k_2, k_3, k_4$ and $t$ satisfying $m>t$, $k_1 \geq k_2+1 \geq 2$ 
and $k_4 \geq k_3 \geq 3$, if we let $k = 2 k_1  + (k_2+1)(k_3-3) + (k_4-k_3) + 1$, then
\begin{align*}
\Delta_{r,s,k}(m) \leq & \left( 1+\frac{m}{t} \right) 
(\Delta_{r,s,k}(t) + 2\Delta_{r,s-1,k_1}(t) + \Delta_{r,s-2,k_2}(t)) + \\
& + \Pi'_{r,s,k_3}\left( 1+\frac{m}{t} \right) + \Delta_{r,s,k_4}\left( 1+\frac{m}{t} \right) 
\qquad \textrm{for $s \geq 4$ and} \\
\Delta_{r,s,k}(m) \leq & \left( 1+\frac{m}{t} \right) 
(\Delta_{r,s,k}(t) + 2\Delta_{r,s-1,k_1}(t) + r-1) + \\
& + \Pi'_{r,s,k_3}\left( 1+\frac{m}{t} \right) + \Delta_{r,s,k_4}\left( 1+\frac{m}{t} \right) 
\qquad  \textrm{for $s = 3$}.
\end{align*}
\end{recurrence}

\begin{proof}
Consider a $\DFF_{r,s,k}(m)$-matrix $M$.
We partition the rows of $M$ into $b:=\lceil m/t \rceil \leq m/t+1$ layers
$L_1, \dots, L_b$ of at most $t$ consecutive rows each.

A column is
\begin{itemize}
\item \emph{local} in layer $L_i$ if all its $1$'s appear in layer $L_i$,
\item \emph{top-concentrated} in layer $L_i$ if it has at least $k_1 + 1$ $1$'s in layer $L_i$ and at 
least one $1$-entry below $L_i$,
\item \emph{bottom-concentrated} in layer $L_i$ if it has at least $k_1 + 1$ $1$'s in layer $L_i$ and at 
least one $1$-entry above $L_i$,
\item \emph{middle-concentrated} in layer $L_i$ if it has at least $k_2 + 2$ $1$'s in layer $L_i$ and at 
least one $1$-entry above and one below layer $L_i$,
\item \emph{doubly-scattered} if it has at least two $1$'s in at least $k_3$ layers,
\item \emph{scattered} if it has a $1$-entry in at least $k_4$ layers.
\end{itemize}

These categories are analogous to those used by Nivasch~\cite{Nivasch10}, except that we added the category
of doubly-scattered columns. This allows us to use $\Pi'$ instead of $\Delta$ in one summand of the recurrence.
As one of the consequences, when $s\geq 6$, the upper bound on the maximum number of $1$'s in a doubled 
$(r,s)$-formation-free $n \times n$ matrix in Lemma~\ref{lem:fatff} is similar to the best known upper 
bound on $F_{r,s}(n)$, although it is closer to $F_{r,s+1}(n)$ when $s=3$.

Every column falls into one of these categories. If a column is in none of them, then its number of $1$'s 
is maximized when it has $k_1$ $1$'s in its top and bottom nonzero layers, $k_2+1$ $1$'s in some other $k_3-3$ 
layers and a single $1$ in some additional $k_4-k_3$ layers. Thus it contains only at most
$2 k_1 + (k_2+1)(k_3-3) + (k_4-k_3) \leq k-1$ $1$-entries.

For each layer $L_i$, the number of columns local in $L_i$ is at most $\Delta_{r,s,k}(t)$. 
For every fixed $i$ we consider the columns that are top-concentrated in $L_i$ and let $M'_i$ be 
the submatrix of $M$ defined by these columns and the rows of $L_i$. Let $M_i$ be obtained from
$M'_i$ by removing the lowest $1$-entry from every column. If $M_i$ contains a doubled $(r,s-1)$-formation,
then $M$ contains a doubled $(r,s)$-formation. Thus there are at most $\Delta_{r,s-1,k_1}(t)$ columns
top-concentrated in $L_i$. Similarly, there are at most $\Delta_{r,s-1,k_1}(t)$ columns
bottom-concentrated in $L_i$. For $s \geq 4$, there at most $\Delta_{r,s-2,k_2}(t)$ columns 
middle-concentrated in $L_i$.
For $s=3$, there are at most $r-1$ columns middle concentrated in $L_i$, because an $r$-tuple of columns
with two $1$'s in layer $L_i$ and at least one $1$ above and one below contains a doubled $(r,3)$-formation.

To bound the number of doubly-scattered columns, we \emph{contract} each layer into a single row. 
That is, we write $1$ for every column containing at least two $1$'s in the layer and $0$ otherwise.
If there is an $(r,s)$-formation on the contracted doubly-scattered columns, then $M$ contains 
a doubled $(r,s)$-formation. Thus, by Lemma~\ref{lem:form2spl}, there are at most 
$\Pi'_{r,s,k_3}(\lceil m/t \rceil)$ doubly-scattered columns.
By a similar argument, the number of scattered columns is at most $\Delta_{r,s,k_4}(\lceil m/t \rceil)$. The only 
difference is that while contracting, we write $1$ for the columns containing at least one $1$ in the 
layer.
\end{proof}

Similarly to Nivasch's functions $R_s(d)$, we define a hierarchy of functions $D_s(d)$, where $s\geq 1$
and $d \geq 2$, as follows:
$D_1(d)=0$, $D_2(d)=2$, $D_s(2)=2^{s-1}+2^{s-2}-1$ and
when $s,d \geq 3$
\[
D_s(d) = 2 D_{s-1}(d) + (D_{s-2}(d)+1)(R_s(d-1)-3) + D_s(d-1)-R_s(d-1) + 1.
\]
Then
\begin{align*}
D_3(d) &= 2d+1, \qquad D_4(d) \le O(d^2), \qquad D_5(d) \le O(d 2^d), \\
D_{2t+3}(d) &\le 2^{(1/t!)d^t + O(d^{t-1})} \quad \textrm{for $t\geq 2$ and} \quad
D_{2t+4}(d) \le 2^{(1/t!)d^t \log(d) + O(d^t)} \quad \textrm{for $t\geq 1$.} \\
\end{align*}

\begin{corollary}
\label{cor:dblffhier}
For every $d\geq 2$, $s \geq 2$, $r\geq 2$, $m$ and $k$ satisfying $m \geq k \geq D_s(d)$ we have
\[
\Delta_{r,s,k}(m) \leq c_s r m \alpha_d(m)^{s-2},
\]
where $c_s$ is a constant depending only on $s$.
\end{corollary}

\begin{proof}
The proof proceeds by induction on $d, s$ and $m$ similarly to the proof of 
\cite[Corollary~4.12]{Nivasch10}.
In the case $s=2$, we apply Observation~\ref{obs:drpolys} and so the lemma holds with $c_2=1$.
For every $s \geq 3$ let $m_0(s)$ be a constant such that
\[
m \geq 1 + (6s)^s \lceil \log_2(m)\rceil^{s^2} \qquad \textrm{for every $m \geq m_0(s)$}.
\]
Let $\widehat{c}_1=\widehat{c}_2=1$ and for $s \geq 3$ we define $\widehat{c}_s$ in the order of increasing $s$ as
\[
\widehat{c}_s := \max\{c'_s, \bar{c}_s, 9 \widehat{c}_{s-1}, 9 \widehat{c}_{s-2}, m_0(s)^{s-1}\},
\]
where $\bar{c}_s$ is the constant from Corollary~\ref{cor:dblffd2} and $c'_s$ is the constant
from Lemma~\ref{lem:nivasch514}.
For every $s \geq 3$ and $d\geq 2$, we define a function $\bar{\alpha}_{d,s}$ by 
$\bar{\alpha}_{2,s}(m) = \lceil \log(m) \rceil$, $\bar{\alpha}_{d,s}(m) = 1$ if $m \leq m_0(s)$ and 
\[
\bar{\alpha}_{d,s}(m) = 1 + \bar{\alpha}_{d,s}(6 s \bar{\alpha}_{d-1,s}(m)^{s-2}) \qquad \textrm{otherwise.}
\]
Then $\bar{\alpha}_{d,s}(m)$ is well defined and differs by at most an additive constant (depending on $s$) 
from the values of the $d$th inverse Ackermann function $\alpha_d(m)$ for all $s$, $d$ and $m$ 
(this can be shown similarly to~\cite[Appendix~C]{Nivasch10}). 
The functions also satisfy $\bar{\alpha}_{d,s}(m) \geq  \bar{\alpha}_{d,s-1}(m)$.
It is thus enough to prove
\[
\Delta_{r,s,k}(m) \leq \widehat{c}_s rm \bar{\alpha}_{d,s}(m)^{s-2}.
\]
The case $d=2$ follows from Corollary~\ref{cor:dblffd2}. 
The cases $m\leq m_0(s)$ follow from Observation~\ref{obs:drpolys}.
Now $s \geq 3$, $d \geq 3$ and $m > m_0(s)$. We apply Recurrence~\ref{rec:rec2} with:
\begin{align*}
k_1 & = D_{s-1}(d), \qquad k_2 = D_{s-2}(d), \qquad k_3 = R_{s}(d-1),  \\
k_4 & = D_{s}(d-1), \qquad k =D_s(d) \qquad \textrm{and} \qquad t = 6 s \bar{\alpha}_{d-1,s}(m)^{s-2}.
\end{align*}
By the induction hypothesis,
\begin{align*}
2\Delta_{r,s-1,k_1}(t) + \Delta_{r,s-2,k_2}(t) 
&\leq r \frac{\widehat{c}_s}{3} t \bar{\alpha}_{d,s}(m)^{s-3} \qquad \textrm{when $s \geq 4$,} \\
2\Delta_{r,s-1,k_1}(t) + r-1 
&\leq r \frac{\widehat{c}_s}{3} t \bar{\alpha}_{d,s}(m)^{s-3} \qquad \textrm{when $s = 3$,} \\
\Delta_{r,s,k_4}\left( 1+\frac{m}{t} \right) 
&\leq \widehat{c}_s r \frac{2m}{t} \bar{\alpha}_{d-1,s}(m)^{s-2}
\leq r \frac{\widehat{c}_s}{3s} \leq r \frac{\widehat{c}_s}{9} m \bar{\alpha}_{d,s}(m)^{s-3} 
~\textrm{for $s \geq 3$}
\end{align*}
and by Lemma~\ref{lem:nivasch514},
\[
\Pi'_{r,s,k_3}\left( 1+\frac{m}{t} \right) 
\leq r s \widehat{c}_s \frac{2 m}{t}  \bar{\alpha}_{d-1,s}(m)^{s-3} 
\leq r \frac{\widehat{c}_s}{3} m 
\leq r \frac{\widehat{c}_s}{3} m \bar{\alpha}_{d,s}(m)^{s-3}.
\]
Substituting into Recurrence~\ref{rec:rec2} we get
\begin{align*}
\Delta_{r,s,k}(m) & \leq \frac{m}{t} \Delta_{r,s,k}(t) + \Delta_{r,s,k}(t) + 
(m+t) r \frac{\widehat{c}_s}{3} \bar{\alpha}_{d,s}(m)^{s-3} + \frac{4\widehat{c}_s}{9} r m \bar{\alpha}_{d,s}(m)^{s-3} \\
& \leq \frac{m}{t} \Delta_{r,s,k}(t) + \frac{7 \widehat{c}_s}{9} r m \bar{\alpha}_{d,s}(m)^{s-3} + 
\Delta_{r,s,k}(t) + \frac{\widehat{c}_s}{3} r t \bar{\alpha}_{d,s}(m)^{s-3}.
\end{align*}
By Observation~\ref{obs:drpolys}, $\Delta_{r,s,k}(t) \leq r t^{s-1} \leq r (6 s \bar{\alpha}_{d-1,s}(m)^{s-2})^{s-1}$,
which is at most $rm$, because $m \geq m_0(s)$. So $\Delta_{r,s,k}(t) \leq \widehat{c}_s rm/9$.  
Similarly $t \bar{\alpha}_{d,s}(m)^{s-3} \leq m/3$. Thus
\begin{align*}
\Delta_{r,s,k}(m) & \leq \frac{m}{t} \Delta_{r,s,k}(t) + \frac{7 \widehat{c}_s}{9} r m \bar{\alpha}_{d,s}(m)^{s-3} 
+ \frac{\widehat{c}_s}{9}rm + \frac{\widehat{c}_s}{9} r m \\
& \leq \frac{m}{t} \Delta_{r,s,k}(t) + \widehat{c}_s rm \bar{\alpha}_{d,s}(m)^{s-3} \\
& \leq \frac{m}{t} \widehat{c}_s rt \bar{\alpha}_{d,s}(t)^{s-2} + \widehat{c}_s rm \bar{\alpha}_{d,s}(m)^{s-3} 
\qquad \textrm{by the induction hypothesis} \\
& \leq m \widehat{c}_s r \cdot ((\bar{\alpha}_{d,s}(m)-1)^{s-2} + \bar{\alpha}_{d,s}(m)^{s-3}) \\
& \leq \widehat{c}_s rm \bar{\alpha}_{d,s}(m)^{s-2}.
\qedhere
\end{align*}
\end{proof}

Let $\beta_s(m) := D_s(\alpha(m))$.

\begin{corollary}
\label{cor:dblff}
An $m \times n$ matrix $M$ with at least $\beta_s(m)$ $1$-entries in every column contains 
a doubled $(\lfloor (n-1)/(m c'_s)\rfloor, s)$-formation, where $c'_s$ is a constant depending only on $s$.
\end{corollary}

\begin{proof}
Let $c'_s = c_s 3^{s-3}$, where $c_s$ is the constant from Corollary~\ref{cor:dblffhier} and let 
$r = \lfloor (n-1)/(m c'_s)\rfloor$.
If $M$ did not contain a doubled $(r,s)$-formation, its number of columns would be, by Corollary~\ref{cor:dblffhier} with $d=\alpha(m)$, at most 
\[
c_s r m \alpha_{\alpha(m)}(m)^{s-3} \leq r m c_s 3^{s-3} = \lfloor (n-1)/(m c'_s)\rfloor m c'_s < n.
\qedhere
\]
\end{proof}

A set $S$ of $Brs$ $1$-entries forms a \emph{$B$-fat $(r,s)$-formation} in $M$ if there exists an $s$-partition
of the rows and an $r$-tuple of columns each of which has $B$ $1$-entries of $S$ in each interval. 
A matrix $M$ is \emph{$B$-fat $(r,s)$-formation-free} if it contains no $B$-fat $(r,s)$-formation.

We now prove a more precise version of Lemma~\ref{lem:fatff}.

\begin{lemma}
\label{lem:fatffprecise}
For all positive integers $m,n,s$ and $B$, an $m \times n$ matrix $M$ 
with at least $2(\beta_s(m)+2)Bn$ $1$-entries contains
a $B$-fat $(\lfloor nB/(m c_s)\rfloor, s)$-formation, where $c_s$ is a constant depending only on $s$.
\end{lemma}

\begin{proof}
We transform the given matrix $M$ to a matrix $\overline{M}$ with the same number of $1$-entries in every column
using the idea from the proof of Lemma~4.1 from~\cite{Nivasch10}.
Let $v(q)$ be the number of $1$-entries in a column $q$ of $M$. In every column $q$, we split the $1$-entries 
into chunks of consecutive $(\beta_s(m)+2)B$ $1$'s. The last less than $(\beta_s(m)+2)B$ $1$'s are discarded. 
Each of the chunks gets its own column with $1$-entries in the rows where the $1$-entries of the chunk lie. 
These columns form the matrix $\overline{M}$. 
Note that the order in which the columns are placed to $\overline{M}$ is not important. 
Because we discarded at most $(\beta_s(m)+2)Bn$ $1$'s and every column of 
$\overline{M}$ has exactly $(\beta_s(m)+2)B$ $1$'s, $\overline{M}$ has at least $n$ columns.
Observe that for every $r$ and $s$ if $\overline{M}$ contains a $B$-fat $(r,s)$-formation, then so does $M$.

We consider only the first $n$ columns of $\overline{M}$. We also remove at most $B-1$ rows so as to have the 
number of rows divisible by $B$. We still have at least $(\beta_s(m)+1)B$ $1$'s in every column. 
In each column $q$, we select a set $S$ of $1$'s such
that none of them is among the first or the last $B-1$ $1$'s of the column $q$ and there are at least $B-1$
$1$'s between every two $1$'s of $S$. We take $S$ of size $\beta_s(m)$ and remove all the other $1$'s in $q$.
The rows of $\overline{M}$ are now grouped into intervals of rows $\{iB+1, \dots, (i+1)B\}$. By the choice of $S$,
every column contains at most one $1$-entry in every interval. We obtain $\overline{\overline{M}}$ by 
contracting each of the intervals of rows into a single row. 

The matrix $\overline{\overline{M}}$ has $\lfloor m/B \rfloor$ rows and $n$ columns, each of them having
$\beta_s(m)$ $1$'s. It thus contains a doubled $(\lfloor (n-1)B/(m c'_s)\rfloor, s)$-formation by 
Corollary~\ref{cor:dblff}. By the choice of $S$, this implies that $\overline{M}$ and consequently $M$ 
contain a $B$-fat $(\lfloor (n-1)B/(m c'_s)\rfloor, s)$-formation.
\end{proof}

\begin{proof}[Proof of Lemma~\ref{lem:fatff}]
Let $\zeta_s(m) = 2(\beta_s(m)+2) \max\{1, c_s\}$, where $c_s$ is the constant from Lemma~\ref{lem:fatffprecise}.
Let $M$ be an $m \times n$ matrix with at least $\zeta_s(m) B n$ $1$-entries.
By Lemma~\ref{lem:fatffprecise}, $M$ contains a $B$-fat $(\lfloor nB/m\rfloor, s)$-formation.
\end{proof}

\subsection{Sets of permutations with bounded VC-dimension}
\label{sub:uperm}
In this section we prove Theorem~\ref{thm:vcub}.
It will be more convenient for the proof to substitute the permutations by their corresponding permutation matrices. 
That is, we have a set $\mathcal{P}$ of $n$-permutation matrices and for every $(k+1)$-tuple $(a_1, \dots, a_{k+1})$ 
of columns, there is a forbidden $(k+1)$-permutation matrix $S_{a_1,\dots,a_{k+1}}$.

Let $M_{\mathcal{P}}$ be a $(0,1)$-matrix with $1$-entries on the positions where at least one matrix 
from $\mathcal{P}$ has a $1$-entry. 
Let $|M|$ denote the number of $1$-entries in a $(0,1)$-matrix $M$ and let 
$v(\mathcal{P})=v(M_{\mathcal{P}})=|M_{\mathcal{P}}|/n$.
Similarly to Raz's proof of the exponential upper bound on $r_2(n)$~\cite{Raz00}, we will remove matrices 
from $\mathcal{P}$ until we decrease $v(\mathcal{P})$ below some threshold $T(n)$.
When $v(\mathcal{P}) \leq T(n)$, then $|\mathcal{P}| \leq T(n)^n$ since 
the number of permutation matrices contained in $M_{\mathcal{P}}$ is bounded from above by the maximum of a product 
of $n$ numbers with sum $v(\mathcal{P}) n$.

Let $\gamma_k(n) = 2 (k+1)! \zeta_{k+1}(n)$, where $\zeta_{k+1}(n)$ is the function from 
Lemma~\ref{lem:fatff}. 

\begin{lemma} 
\label{lem:iter}
Let $\mathcal{P}$ be a set of $n$-permutation matrices with VC-dimension $k$ such that 
$v(\mathcal{P}) \geq 2 \gamma_k(n)$. Then there is a set $\mathcal{P}' \subset \mathcal{P}$ satisfying
\begin{align*}
v(\mathcal{P}') & \leq v(\mathcal{P}) - \frac{v(\mathcal{P})^2}{\gamma_k^2(n) n} \\
|\mathcal{P}'| & \geq \frac{|\mathcal{P}|}{2 v(\mathcal P)^{k}}.
\end{align*}
\end{lemma}

\begin{proof}
Let $B:=\lfloor v(\mathcal{P})/ \zeta_{k+1}(n) \rfloor$.
By Lemma~\ref{lem:fatff}, the matrix $M_{\mathcal{P}}$ contains a $B$-fat $(B, k+1)$-formation. 
Let $C$ be the set of the $B$ columns of the formation and let $\mathcal{R} = \{R_1, R_2, \dots, R_{k+1}\}$ be 
the intervals of rows of the formation.

Consider some $t$-tuple $Q=\{q_1, \dots, q_t\}$ of columns from $C$, where $t\leq k+1$.
Let $\mathcal{I}_Q$ be the set of all injective functions $I:Q\rightarrow \mathcal{R}$ 
assigning the intervals $R_j$ to the columns $q_i$. 
We say that a permutation matrix $P$ \emph{obeys} $I\in \mathcal{I}_Q$ if for every $i \in \{1,2, \dots, t\}$, 
the $1$-entry of $P$ in the column $q_i$ lies in some row of $I(q_i)$.
For each $I\in \mathcal{I}_Q$ let $\mathcal{P}_I$ be the set of matrices $P \in \mathcal{P}$ that obey $I$. 
The $t$-tuple $Q$ of columns is said to be \emph{criss-crossed} if
\[
\forall I \in \mathcal{I}_{Q}: |\mathcal{P}_I| \geq |\mathcal{P}|/v(\mathcal{P})^{t}.
\]

Suppose that some $(k+1)$-tuple of columns from $C$ is criss-crossed.
Then every $(k+1)$-permutation appears as a restriction of some matrix from $\mathcal{P}$ on the criss-crossed 
$(k+1)$-tuple of columns. Hence the VC-dimension of $\mathcal{P}$ is at least $k+1$.

Consequently, there is some $t$ such that $0\leq t \leq k$ and the largest criss-crossed set $Q$ 
of columns from $C$ has size $t$.
This means that for every column $u$ outside $Q$, we can find an injective function 
$J_{u}\in\mathcal{I}_{Q\cup\{u\}}$ 
such that $|\mathcal{P}_{J_{u}}| < |\mathcal{P}|/v(\mathcal{P})^{t+1}$.
On the other hand, if we restrict $J_{u}$ on $Q$, the resulting function 
$I_{u}:=J_{u}\restriction{Q}$ satisfies $|\mathcal{P}_{I_{u}}| \geq |\mathcal{P}|/v(\mathcal{P})^{t}$. 
To each choice of $u \in C \setminus Q$, we assign the function $I_{u} \in \mathcal{I}_{Q}$
and the interval $J_{u}(u)$ of rows.
Some function $I \in \mathcal{I}_{Q}$ was then assigned to at least $(|C|-k)/|\mathcal{I}_{Q}|$ columns.
Because $B \geq 4 (k+1)!$, we have 
\[
 \frac{|C|-k}{|\mathcal{I}_{Q}|} \geq \frac{B-k}{(k+1)!} \geq \frac{v(\mathcal{P})}{2\zeta_{k+1}(n)(k+1)!} \geq \frac{v(\mathcal{P})}{\gamma_k(n)}.
\]
Let $\mathcal{T}_I$ be the set of some $\lceil v(\mathcal{P}) / \gamma_k(n) \rceil$ columns that 
were assigned the function $I$. Because $v(\mathcal{P}) \geq 2 \gamma_k(n)$, we have
\begin{equation}
\label{eq:tiBounds}
\frac{v(\mathcal{P})}{\gamma_k(n)} \leq |\mathcal{T}_I| \leq \frac{2v(\mathcal{P})}{\gamma_k(n)} \leq \frac{v(\mathcal{P})}{2}.
\end{equation}

For each column $q_i \in Q$, we remove from $M$ all $1$-entries in 
the column $q_i$ except those that lie in the rows of $I(q_i)$. 
This reduces the number of permutation matrices, but there are still at least 
$|\mathcal{P}|/v(\mathcal{P})^{t}$ 
of them. Then we remove from $M$ the $1$-entries in each column $u \in \mathcal{T}_I$ that lie
in the set of rows $J_{u}(u)$. See Fig.~\ref{fig:crisscrossed}. 
Thus we removed at least $B$ $1$-entries from each of these columns.
The removed $1$-entries of each of these columns decreased the number of permutation matrices by
at most $|\mathcal{P}_{J_u}| \leq |\mathcal{P}|/v(\mathcal{P})^{t+1}$.

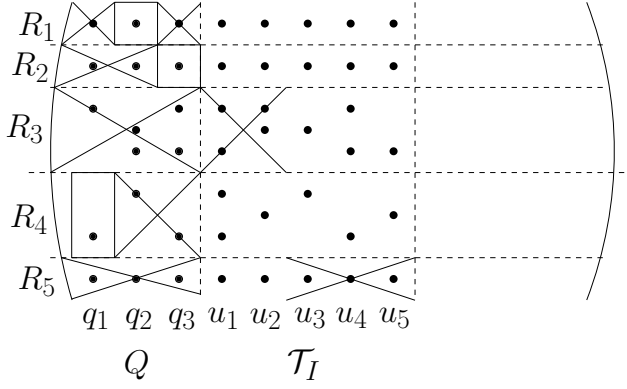
\myfig{crisscrossed}{0.5}{A criss-crossed set $Q$ of $3$ columns and a set $\mathcal{T}_I$ of columns 
$\{u_1, \dots, u_5\}$, where $I(q_1)=R_4$, $I(q_2)=R_1$, $I(q_3)=R_2$, $J_{u_1}(u_1)=J_{u_2}(u_2)=R_3$ and 
$J_{u_3}(u_3)=J_{u_4}(u_4)=J_{u_5}(u_5)=R_5$. The $1$-entries from the crossed rectangles are removed.
}

Let $\mathcal{P}' \subset \mathcal{P}$ be the set of permutation matrices containing none of the removed $1$-entries.
Using the bounds from Equation~\eqref{eq:tiBounds}, we obtain
\begin{align*}
|M_{\mathcal{P}'}| & \leq |M_{\mathcal{P}}| - B |\mathcal{T}_I|
\leq n v(\mathcal{P}) - \frac{v(\mathcal{P})^2}{\gamma_k^2(n)}, \\
|\mathcal{P}'| & \geq 
\frac{|\mathcal{P}|}{v(\mathcal{P})^{t}} - 
\frac{|\mathcal{P}| |\mathcal{T}_I|}{v(\mathcal{P})^{t+1}}
\geq \frac{|\mathcal{P}|}{2 v(\mathcal{P})^{t}}
\geq \frac{|\mathcal{P}|}{2 v(\mathcal{P})^{k}}. 
\qedhere
\end{align*}
\end{proof}

\begin{proof}[Proof of Theorem~\ref{thm:vcub}]
Let $\mathcal{P}$ be a set of permutation matrices with VC-dimension $k$.
We will bound its size by iteratively applying Lemma~\ref{lem:iter}. Let $\mathcal{P}_0 = \mathcal{P}$
and for $j\geq 1$ let $\mathcal{P}_{j}$ be the $\mathcal{P}'$ given by the lemma applied on $\mathcal{P}_{j-1}$.

The iterations are further grouped into phases. Let $\phi_0 := 0$. Phase $i$ ends after the first iteration
$\phi_i$ after which $v(\mathcal{P}_{\phi_i}) \leq v(\mathcal{P}_{\phi_{i-1}})/2$. 
Let $v_i := v(\mathcal{P}_{\phi_i})$. Then an iteration of phase $i$ is applied on a set $\mathcal{P}$ of
permutations satisfying
\begin{equation} \label{eq:vbounds}
\frac{v_{i-1}}{2} \leq v(\mathcal{P}) \leq v_{i-1}
\end{equation}
Further, let 
\begin{equation} \label{eq:defT}
T := \gamma^2_k(n) \log(\gamma_k(n)).
\end{equation}
We end after the first phase $l$ satisfying $v_l \leq 2 T$. We thus have 
\begin{equation} \label{eq:plast}
 |\mathcal{P}_{\phi_l}| \leq (2T)^n.
\end{equation}
For every $i \geq 1$ we have 
\begin{equation} \label{eq:vli}
v_{l-i} \geq 2^i T.
\end{equation}

We now count the number of iterations in phase $i$. By Lemma~\ref{lem:iter} and~(\eqref{eq:vbounds}), 
each of these iterations decreases $v(\mathcal{P})$ by at least $v^2_{i-1}/(4\gamma_k^2(n) n)$.
Therefore the phase ends after at most $\lceil 2 \gamma_k^2(n) n / v_{i-1} \rceil \leq 3 \gamma_k^2(n) n / v_{i-1}$
iterations.
Consequently
\begin{align*}
 |\mathcal{P}_{\phi_{i-1}}| & \leq 
|\mathcal{P}_{\phi_i}| \cdot \left(2 v_{i-1}^{k}\right)^{3 \gamma_k^2(n) n / v_{i-1}} \qquad \qquad 
\textrm{by Lemma~\ref{lem:iter} and \eqref{eq:vbounds}}\\
& \leq |\mathcal{P}_{\phi_i}| \cdot 2^{(1+k \log v_{i-1})3 \gamma_k^2(n) n / v_{i-1}} \qquad \\
& \leq |\mathcal{P}_{\phi_i}| \cdot 2^{6k \gamma_k^2(n) n \log v_{i-1} / v_{i-1}} \qquad 
\end{align*}
and
\begin{align*}
|\mathcal{P}| &= |\mathcal{P}_0| \leq |\mathcal{P}_{\phi_l}|  
\prod_{i=0}^{l-1} 2^{6k \gamma_k^2(n) n \log v_{i} / v_{i}} \\
& \leq |\mathcal{P}_{\phi_l}| 2^{6k \gamma_k^2(n) n \sum_{i=1}^{l} \log (2^i T) / (2^i T)} 
\quad \textrm{by~\eqref{eq:vli} and since $\frac{\log(x)}{x}$ is decreasing on $[2T, \infty)$} \\
& \leq (2T)^n \cdot 2^{6k \gamma_k^2(n) n (2+ \log T)/T} \qquad \textrm{by~\eqref{eq:plast}}\\
& \leq \left( 2\gamma_k^2(n)\log(\gamma_k(n)) \right)^n \cdot 2^{30kn} \qquad \textrm{by~\eqref{eq:defT}}.
\end{align*}
Since $\gamma_k(n) \in O(\zeta_{k+1}(n))$, we have
\begin{align*}
r_3(n) &\leq \left(O(\alpha(n)^4 \log(\alpha(n)))\right)^n, \\
r_4(n) &\leq 2^{n \cdot \left(2 \alpha(n) + 3 \log(\alpha(n)) + O(1) \right)}, \\
r_{2t+2}(n) &\leq 2^{n \cdot \left((2/t!)\alpha(n)^t + O(\alpha(n)^{t-1}) \right)} \qquad \textrm{for $t\geq 1$ and} \\
r_{2t+3}(n) &\leq 2^{n \cdot \left((2/t!)\alpha(n)^t \log(\alpha(n)) + O(\alpha(n)^t) \right)} \qquad \textrm{for $t\geq 1$.} 
\qedhere
\end{align*}
\end{proof}

\begin{remark*}
\rm
Let an $n$-function be a total function $f:[n] \rightarrow [n]$.
We say that a set $\mathcal F$ of $n$-functions has \emph{VC-dimension with respect to permutations} 
(abbreviated as \emph{pVC-dimension}) $k$ if $k$ is the largest integer 
such that the set of restrictions of the functions in $\mathcal F$ to some $k$-tuple of
elements from $[n]$ contains all $k$-permutations. Let $r'_k(n)$ be the size of the largest set of 
$n$-functions with pVC-dimension $k$.
Observe that the proofs of this section never use the fact that the matrices in $\mathcal{P}$ have exactly one $1$-entry
in every row. Thus the upper bound from Theorem~\ref{thm:vcub} also holds with $r'_k(n)$ in place of $r_k(n)$.
\end{remark*}

\section{Lower bounds}
\label{sec:lb}

\subsection{Matrices from sequences}
\label{sub:lseq}
Let $DS_{s}$ be the $s \times 2$ matrix with $1$-entry in the $i$th row and $j$th column exactly when $i+j$
is odd. For example
\[
DS_4=
\left(\begin{smallmatrix} 
& \bullet \\ 
\bullet & \\ 
& \bullet \\ 
\bullet & \\ 
\end{smallmatrix}\right)
\qquad
\textrm{and}
\qquad
DS_5=
\left(\begin{smallmatrix} 
& \bullet \\ 
\bullet & \\ 
& \bullet \\ 
\bullet & \\ 
&\bullet & \\ 
\end{smallmatrix}\right).
\] 
Based on a construction of Davenport--Schinzel sequences of order $3$ and length $\Omega(n \alpha(n))$ by 
Hart and Sharir~\cite{HartSharir}, F\"{u}redi and Hajnal~\cite{FurediHajnal} constructed $n \times n$ 
$DS_4$-avoiding matrices $A_n$ with $\Omega(n \alpha(n))$ $1$-entries.
We will use a different construction of $\mathrm{DS}(s)$-sequences of orders $s=3$ and all even $s\geq 4$ 
by Nivasch~\cite{Nivasch10} that together with the following transcription will provide us with 
$DS_{s+1}$-avoiding matrices with the additional property of having the same number of $1$-entries
in every column.

Let $S$ be a sequence over $n$ symbols that can be partitioned into $m$ blocks. Recall that
each block contains only distinct symbols. 
We number the symbols $1, \dots, n$ in the increasing order of their first appearance.
The \emph{sequence$\rightarrow$matrix transcription of $S$, $\SMT(S)$}, is the $m \times n$ 
$(0,1)$-matrix with a $1$-entry in the $i$th row and $j$th column exactly if
the $i$th block in the sequence contains the symbol $j$.

\begin{observation}{\rm (\cite{FurediHajnal})}
\label{obs:onelongerseq}
If $S$ is a sequence avoiding the alternating pattern $aba\dots$ of length $s+2$, 
then $\SMT(S)$ avoids $DS_{s+1}$.
\end{observation}
\begin{proof}
If $\SMT(S)$ contains $DS_{s+1}$, then $S$ contains the alternating sequence $ba\dots$ of length $s+1$
for some $a<b$. By the numbering of the symbols, the symbol $a$ appears before the first occurrence of
$b$, therefore $S$ contains $aba\dots$ of length $s+2$ and thus $S$ is not a $\mathrm{DS}(s)$-sequence.
\end{proof}

\begin{lemma}
\label{lem:constrds3}
For every $n$ there exists an $n \times n$ $DS_{4}$-avoiding matrix $M_n$ with at least
$2\alpha(n) - O(1)$ $1$-entries in every column.
\end{lemma}

\begin{proof}
Let $A_d(x)$ be the $d$th function of the Ackermann hierarchy. We refer the reader to 
Nivasch's paper~\cite{Nivasch10} for the definition. Let $A(x) = A_x(3)$ be the Ackermann
function.

In Section~6 of~\cite{Nivasch10}, Nivasch constructs for every $d,m \geq 1$ an $ababa$-free sequence $Z_d(m)$.
We use the sequences $Z'_d = Z_d(8d+4)$ which have the following properties:
\begin{itemize}
\item Each symbol appears exactly $2d+1$ times.
\item The sequence can be decomposed into blocks of average length at least $4d+2$ (by~\cite[Lemma~6.2]{Nivasch10}).
\item The number $N_d$ of symbols of the sequence is at most $A_d(8d+4+c)$ (by~\cite[Lemma~6.2]{Nivasch10}), 
where $c$ is an absolute constant. 
\end{itemize}
Let $M_d$ be the number of blocks of $Z'_d$. By counting the length of $Z'_d$ in two ways, $(2d+1)N_d \geq (4d+2)M_d$
and thus $N_d \geq 2M_d$.
By the analysis before Equation~(35) in~\cite{Nivasch10}, there is some $d_0$ such that for 
$d \geq d_0$ we have
\begin{align*}
N_d &\leq A_d(8d+4+c) \leq A_d(A(d+1)) = A(d+2) \quad \textrm{and so} \\
\alpha(N_d) &\leq d+2.
\end{align*}

Then $\SMT(Z'_d)$ is an $M_d \times N_d$ matrix with $2d+1 \geq 2\alpha(N_d)-3$ $1$-entries in every column.
By Observation~\ref{obs:onelongerseq}, $\SMT(Z'_d)$ avoids $DS_4$. 
We construct the matrix $M_{N_d}$ by adding empty rows to $\SMT(Z'_d)$.

For values $n \geq N_{d_0}$ different from $N_d$, we proceed similarly to the method in Section~6 of~\cite{Nivasch10}: 
We consider the largest $N_d$ smaller than $n$ and 
take $\lceil n/N_d \rceil$ copies of $\SMT(Z'_d)$. We place the copies into a single matrix so that
each copy has its own set of consecutive rows and columns. After removing at most half of the 
columns, we obtain a matrix with exactly $n$ columns and at most $n$ rows. The matrix has 
at least $2d+1 \geq 2\alpha(N_{d+1})-5 \geq 2\alpha(n)-5$ $1$-entries in every column. 
The construction of $M_n$ is then finished by adding empty rows to obtain a square matrix.
\end{proof}

\begin{lemma}
\label{lem:constrdseven}
For every $t\geq 1$ and $n$ there exists an $n \times n$ $DS_{2t+3}$-avoiding matrix 
with at least $2^{(1/t!) \alpha(n)^t - O(\alpha(n)^{t-1})}$ $1$-entries in every column. In particular, 
\[
\rmex_{DS_{{2t+3}}}(n) \geq n 2^{(1/t!) \alpha(n)^t - O(\alpha(n)^{t-1})}.
\]
\end{lemma}

\begin{proof}
Let $s:=2t+2$. Since $s$ is even and $s\geq 4$, we can use Nivasch's 
construction~\cite[Section 7]{Nivasch10} of $\mathrm{DS}(s)$-sequences $S^s_k(m)$ with parameters $k,m \geq 2$.
Let $\mu_s(k) := 2^{\binom{k}{(s-2)/2}}$.
We use the sequences $S'_{s,k} = S^s_k(2 \mu_s(k))$, which have the following properties:
\begin{itemize}
\item
Each symbol of $S'_{s,k}$ appears exactly $\mu_s(k)$ times 
(by~\cite[Equation (47)]{Nivasch10}).
\item 
The sequence can be decomposed into blocks of length $2 \mu_s(k)$.
\item
For every $k \geq k_0(s)$, where $k_0(s)$ is a properly chosen constant, the number $N_{s,k}$ of symbols of 
the sequence satisfies $\alpha(N_{s,k}) \leq k+3$ (by~\cite[Equations (50), (51)]{Nivasch10} and analysis similar
to the one in the proof of Lemma~\ref{lem:constrds3}).
\end{itemize}

Let $M_{s,k}$ be the number of blocks of $S'_{s,k}$. It satisfies $2M_{s,k} \leq N_{s,k}$.
The matrix $\SMT(S'_{s,k})$ is a $DS_{s+1}$-avoiding $M_{s,k} \times N_{s,k}$ matrix with at least $\mu_{s,k}$
$1$-entries in every column.
For every $n\geq N_{s,k_0(s)}$ we take the largest $k$ such that $n \geq N_{s,k}$ and proceed in the same 
way as in the proof of Lemma~\ref{lem:constrds3} with $\SMT(S'_{s,k})$ in the place of $\SMT(Z'_d)$.
We have
\[
k \geq \alpha(N_{s,k+1})-4 \geq \alpha(n)-4
\]
and so the number of $1$-entries in every column of the resulting matrix is
\[
\mu_s(k) = 2^{\binom{k}{(s-2)/2}} = 2^{\binom{k}{t}} \geq 2^{(1/t!)k^t - O(k^{t-1})} 
\geq 2^{(1/t!) \alpha(n)^t - O(\alpha(n)^{t-1})}.
\qedhere
\]
\end{proof}

\begin{remark*}
\rm
We could also use the construction of $DS_4$-avoiding matrices with $\Omega(n \alpha(n))$ $1$-entries
by F\"{u}redi and Hajnal~\cite{FurediHajnal}.
The matrices do not have the same number of $1$-entries in every column, but it can be shown that every column 
has at most constant multiple of the average number of $1$-entries per column. 
This would be enough for our purposes.
The base case of the inductive construction in~\cite{FurediHajnal} needs a small fix. 
The matrices $M(s,1)$ and $M(1,s)$ do not satisfy conditions imposed on them. This can be fixed for example 
by taking $\left(\begin{smallmatrix} 1 & 0 \\ 0 & 1 \\ \end{smallmatrix}\right)$ 
for $M(s,1)$ and the matrix with the leftmost column full of $1$-entries and with no $1$-entries in the 
other columns for $M(1,s)$.
\end{remark*}

\subsection{Numbers of $1$-entries in matrices}
\label{sub:lmat}
A matrix is \emph{$k$-full} if some $k$-tuple of its columns contains every $k$-permutation matrix.
The \emph{fullness} of a matrix $A$ is the largest $k$ such that $A$ is $k$-full. In this section
we show a lower bound on the maximum number $p_k(n)$ of $1$-entries in an $n \times n$ matrix 
with fullness $k$. This is achieved by showing that a $k$-full matrix contains the matrix $DS_k$
and applying the results from Section~\ref{sub:lseq}. We prove a slightly stronger statement 
that will be used in the next section.

Let $J_2:=\left(\begin{smallmatrix} & \bullet \\\bullet &  \\ \end{smallmatrix}\right)$.
For an $l$-permutation matrix $P$, we define the \emph{$J_2$-expansion of $P$}, $P^{J_2}$, to be the 
$2l \times 2l$ permutation matrix created by substituting every $1$-entry of $P$ by $J_2$ and every
$0$-entry by a $2 \times 2$ block full of $0$-entries. 

A pair of rows $2i$, $2i+1$ of $P^{J_2}$ will be called \emph{contractible} if the $1$-entry in row $2i$ is to 
the left of the $1$-entry in row $2i+1$. That is, when $\pi^{-1}(i)<\pi^{-1}(i+1)$, where 
$\pi$ is the permutation corresponding to $P$.
To \emph{contract} a pair of rows means to replace them by a single row with $1$-entries in the columns 
where at least one of the two original rows had a $1$-entry.

Let an $(n,m)$-function be a total function $f:[n] \rightarrow [m]$.
A \emph{function matrix} is a $(0,1)$-matrix with exactly one $1$-entry in every column. 
Assigning to a function $f$ a function matrix $G_{f}$ with $G_{f}(i,j)=1 \Leftrightarrow f(j)=i$ provides
a bijection between $(n,m)$-functions and $m \times n$ function matrices.

The set of \emph{$J_2$-expansion flattenings} of $P$ is the set $\mathcal F(P^{J_2})$ of function matrices 
that can be obtained from $P^{J_2}$ by contracting some pairs of contractible rows. Let $\mathcal{P}_l$ be 
the set of $l$-permutation matrices and let
\[
\Phi(l) := \{\mathcal{F}(P^{J_2}): P \in \mathcal{P}_l \}.
\]
For example
\[
\Phi(2) = \left\{ \left\{ 
\left(\begin{smallmatrix}
 &  &  & \bullet \\
 &  & \bullet &  \\
 & \bullet &  &  \\
\bullet &  &  &  \\
\end{smallmatrix}\right)
 \right\},
\left\{
\left(\begin{smallmatrix}
 & \bullet &  &  \\
\bullet &  &  &  \\
 &  &  & \bullet \\
 &  & \bullet &  \\
\end{smallmatrix}\right)
,
\left(\begin{smallmatrix}
 & \bullet &  &  \\
\bullet & &  & \bullet \\
 &  & \bullet &  \\
\end{smallmatrix}\right)
\right\}\right\}
.
\]

\begin{lemma}
\label{lem:todslargeeven}
If an $n \times 2l$ matrix $A$ contains one matrix from $\mathcal F(P^{J_2})$ for every $l$-permutation matrix $P$, 
then $A$ contains an occurrence of $DS_{2l}$ on columns $\{ 2i-1, 2i\}$ for some $i \in [l]$.
\end{lemma}
\begin{proof}
We proceed by induction on $l$. The case $l=1$ is trivial since 
$\Phi(l) = \{\{J_2\}\} = \{\{DS_2\}\}$.

The \emph{$i$th pair of columns} of $A$ is the pair of columns $\{ 2i-1, 2i\}$.
For each $i\leq l$ let $h_i$ be the smallest number such that the $i$th pair of columns of $A$
contains $J_2$ on a subset of rows $\{1 \dots h_i \}$.
Let $t$ be the largest number satisfying $\forall i ~ h_i \leq h_t$.
Let $A^{\setminus t}$ be the $(n-h_t+1) \times 2(l-1)$ matrix obtained from $A$ by removing 
the columns of the $t$th pair, removing the top $h_t-1$ rows and then changing all $1$-entries among the 
first $2(t-1)$ entries in the first row to $0$'s. See Fig.~\ref{fig:lbmatr}

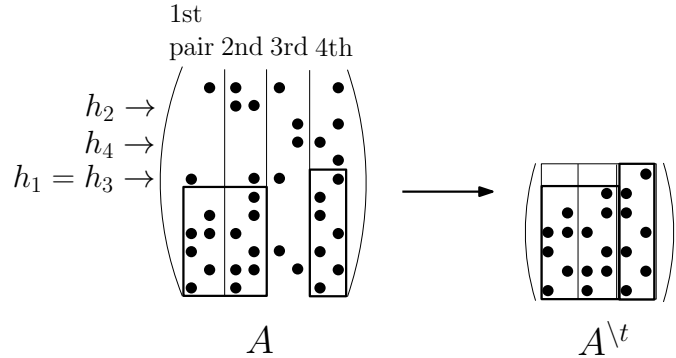
\myfig{lbmatr}{0.7}{Induction step in the proof of Lemma~\ref{lem:todslargeeven}. In this example $t=3$.
}

For every $P$ with the topmost $1$-entry in column $t$, $A$ contains an occurrence of some 
$F \in \mathcal F(P^{J_2})$, that uses the two $1$-entries of the topmost occurrence of $J_2$
on the $t$th pair of columns. These occurrences induce an occurrence of some matrix from every set 
$\mathcal F \in \Phi(l-1)$ in $A^{\setminus t}$.
By the induction hypothesis, $A^{\setminus t}$ contains $DS_{2(l-1)}$ on some $i$th pair of columns.
By the choice of $t$, this occurrence of $DS_{2(l-1)}$ in $A$ does not use any of the rows $\{1 \dots h_i \}$.
Thus we obtain an occurrence of $DS_{2l}$ in $A$.
\end{proof}

For an $l$-permutation matrix $P$ and $i\leq 2l+1$ we define $P^{J_2}(i)$ to be the $(2l+1)$-permutation matrix 
that becomes $P^{J_2}$ after removing the lowest row and column $i$. 
Then $\mathcal F(P^{J_2},i)$ is the set of function matrices that can be obtained from $P^{J_2}(i)$ 
by contracting some pairs of contractible rows. For example
\[
\mathcal F\left(\left( \begin{smallmatrix} \bullet & ~ \\ ~ & \bullet \end{smallmatrix}\right)^{J_2},4 \right) =
\left\{
\left(\begin{smallmatrix}
        & \bullet &         &         &         \\
\bullet &         &         &         &         \\
        &         &         &         & \bullet \\
        &         & \bullet &         &         \\
        &         &         & \bullet &         \\
\end{smallmatrix}\right)
,\quad
\left(\begin{smallmatrix}
        & \bullet &         &         &         \\
\bullet &         &         &         &         \\
        &         &         &         & \bullet \\
        &         & \bullet & \bullet &         \\
\end{smallmatrix}\right)
,\quad
\left(\begin{smallmatrix}
        & \bullet &         &         &         \\
\bullet &         &         &         & \bullet \\
        &         & \bullet &         &         \\
        &         &         & \bullet &         \\
\end{smallmatrix}\right)
,\quad
\left(\begin{smallmatrix}
        & \bullet &         &         &         \\
\bullet &         &         &         & \bullet \\
        &         & \bullet & \bullet &         \\
\end{smallmatrix}\right) \right\}.
\]
Let
\[
\Phi(l,i) := \{\mathcal{F}(P^{J_2},i): P \in \mathcal{P}_l \}.
\]

\begin{lemma}
\label{lem:todslargeodd}
Let $A$ be an $n \times (2l+1)$ matrix and let $A'$ be the matrix obtained from
$A$ by removing the last $1$-entry from each column. If $A'$ contains one matrix 
from $\mathcal F(P^{J_2},i)$ for every $l$-permutation matrix $P$ and every 
$i \in [2l+1]$, then $A$ contains $DS_{2l+1}$.
\end{lemma}
\begin{proof}
For each $i \leq 2l+1$ let $d_i$ be the row number of the lowest $1$-entry in the $i$th column of $A'$.
Let $t$ be any of the rows satisfying $\forall i ~ d_i \geq d_t$.
Let $A^{\setminus t}$ be the $d_t \times 2l$ matrix obtained from $A$ by removing the $t$th column 
and all rows below the $d_t$th row.
Then $A^{\setminus t}$ contains one matrix from every $\mathcal F \in \Phi(l)$, 
therefore by Lemma~\ref{lem:todslargeeven} the matrix $A^{\setminus t}$ contains an occurrence of $DS_{2l}$.
By the choice of $t$, the matrix $A$ contains $DS_{2l+1}$.
\end{proof}

\begin{corollary}
\label{cor:pknlb}
For every $k\geq 1$
\[
p_{k}(n) \geq \rmex_{DS_{k+1}}(n).
\]
\end{corollary}

\begin{proof}
When $k$ is even, the result follows from Lemma~\ref{lem:todslargeodd}, 
since for every $l$-permutation matrix $P$ and for every $i \in [2l+1]$, 
the set $\mathcal F(P^{J_2},i)$ contains some $(2l+1)$-permutation matrix,
namely the matrix without any row contractions.
The result for $k$ odd follows from Lemma~\ref{lem:todslargeeven}.
\end{proof}

The row contractions did not play any role in the proof of Corollary~\ref{cor:pknlb}, 
but they will play a role in Section~\ref{sub:lperm} below.
\begin{corollary}
\label{cor:extrlb}
We have
\begin{align*}
p_3(n) & \geq 2 n \alpha(n) - O(n),  \\
p_k(n) & \geq n 2^{(1/t!)\alpha(n)^t - O(\alpha(n)^{t-1})} \quad \textrm{for $k \geq 4$,}
\end{align*}
where $t := \lfloor (k-2)/2 \rfloor$.
\end{corollary}
\begin{proof}
The lower bound for $k=3$ is by Lemma~\ref{lem:constrds3} and Corollary~\ref{cor:pknlb} and from 
Lemma~\ref{lem:constrdseven} and Corollary~\ref{cor:pknlb} when $k$ is even and $k>3$. 
When $k$ is odd and $k\geq 5$, we use $p_k(n) \geq p_{k-1}(n)$.
\end{proof}

\subsection{Sets of permutations}
\label{sub:lperm}
\begin{proof}[Proof of Theorem~\ref{thm:vclb}]

Given $k$ and $n$, we take the $DS_{k+1}$-avoiding $n \times n$ matrix $A_{k,n}$ from Lemma~\ref{lem:constrds3} 
if $k=3$ or from Lemma~\ref{lem:constrdseven} if $k\geq 4$ is even. Let $\rho_k(n)$ be the number of
$1$-entries that $A_{k,n}$ has in every column, that is $\rho_2(n)= 2\alpha(n) - O(1)$ and
for $t\geq 1$ $\rho_{2t+2}(n) = 2^{(1/t!)\alpha(n)^t - O(\alpha(n)^{t-1})}$.

From $A_{k,n}$ we construct a set of $\rho_k(n)^{n}$ $n \times n$ function matrices by choosing 
some $1$-entry from each column. Then we remove all empty rows, which can make some originally different
function matrices identical. However, the resulting set $\mathcal H$ has size at least 
$\rho_k(n)^{n}/2^n$ as there are at most $2^n$ distinct ways to enlarge a function matrix
by adding empty rows to a matrix with $n$ rows. 

The last step is inflating the rows of the function matrices in $\mathcal H$ into diagonal matrices to 
obtain a set $\mathcal Q$ of $n$-permutation matrices. That is, for every $H \in \mathcal H$, we order the 
$1$-entries primarily by the rows from top to bottom and secondarily from left to right. The permutation 
matrix $Q$ will have $1$-entries on those positions $(i,j)$ such that $H$ has its $i$th $1$-entry in column $j$. 
The reverse process consists of contracting intervals of rows of a permutation matrix $Q$ and we have at most
$2^{n}$ possibilities how to choose the intervals. Thus every permutation matrix can be created 
by expanding at most $2^{n}$ different function matrices.
The size of the set $\mathcal Q$ is
\[
|\mathcal Q| \geq \frac{\rho_k(n)^{n}}{2^{n}2^{n}} = \left( \frac{\rho_k(n)}{4}  \right)^n.
\]
It remains to show that the VC-dimension of $\mathcal Q$ is at most $(k+1)$. We assume for contradiction that
for some $(k+1)$-tuple $C$ of columns and every $(k+1)$-permutation matrix $R$ there exists 
$Q \in \mathcal Q$ that contains $R$ on $C$.

Consider some permutation matrix $Q \in \mathcal Q$ and let $H \in \mathcal H$ be the function 
matrix from which $Q$ was created. 
The matrix $H$ can thus be constructed from $Q$ by contracting some intervals of rows such that 
the restriction of $Q$ on each of these intervals of rows is a diagonal matrix.
So the only change that these contractions can make on an occurrence of $P^{J_2}$ in $Q$ is that 
some pairs of its contractible rows can be contracted. Thus an occurrence of $P^{J_2}$ in $Q$
on the set $C$ of columns can only be created from an occurrence of some 
$F\in\mathcal F(P^{J_2})$ on $C$ in $H$ and in $A_{k,n}$ as well. 
Similarly, an occurrence of $P^{J_2}(i)$ on $C$ in $Q$ implies an occurrence of some matrix 
from $\mathcal F(P^{J_2},i)$ on $C$ in $H$ and $A_{k,n}$. See Fig.~\ref{fig:vclb}.

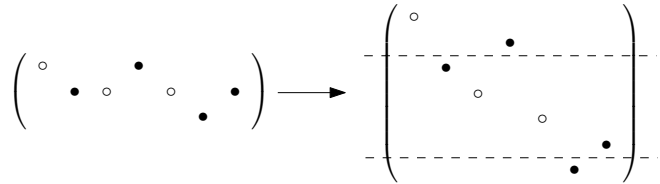
\myfig{vclb}{0.8}{Expansion of an occurrence of a matrix from $\mathcal F(P^{J_2})$.}

Therefore for $k \geq 4$ even, for every $(k/2)$-permutation matrix $P$ and every $i\in [k+1]$,
some matrix from $\mathcal F(P^{J_2},i)$ occurs on $C$ in $A_{k,n}$. Thus, by
Lemma~\ref{lem:todslargeodd}, $A_{k,n}$ contains $DS_{k+1}$, a contradiction.
Similarly if $k=3$, we get a contradiction by Lemma~\ref{lem:todslargeeven}.
\end{proof}

\subsection*{Acknowledgements}
We are grateful to Seth Pettie for pointing us to a way to improve the result of Lemma~\ref{lem:form2spl}.
We also thank Ottfried Cheong, Martin Klazar and Pavel Valtr for inspiring discussions.

\end{document}